\documentclass[a4paper, 12pt, final]{amsart}
\usepackage{amssymb,amsthm,amsmath, amsfonts}
\usepackage{graphicx}
\usepackage{hyperref}
\usepackage{fullpage}

\usepackage{xcolor}
\definecolor{vry}{RGB}{253, 231, 37}

\definecolor{vrg}{RGB}{94,201,98}

\definecolor{vrdg}{RGB}{33, 145, 140}

\definecolor{vrb}{RGB}{59,82,139}

\definecolor{vrp}{RGB}{68,1,84}

\definecolor{vro}{RGB}{249,142,9}

\definecolor{vrr}{RGB}{188,55,84}

\definecolor{vrnb}{RGB}{13,8,135}

\hypersetup{colorlinks=true, linkcolor=vrnb, citecolor=vrp, filecolor=vrr, urlcolor=vrr}

\newtheorem{theorem}{Theorem}[section]
\newtheorem{proposition}[theorem]{Proposition}
\newtheorem{lemma}[theorem]{Lemma}

\newtheorem{corollary}[theorem]{Corollary}

\newtheorem{definition}[theorem]{Definition}

\newtheorem{question}[theorem]{Question}
\newcommand{\round}[1]{\left(#1\right)} % example of author macro
\newcommand{\TV}{\operatorname{TV}}
\newcommand{\Des}{\operatorname{Des}}

\newcommand{\des}{\operatorname{des}}
\newcommand{\Asc}{\operatorname{Asc}}
\newcommand{\asc}{\operatorname{asc}}
\newcommand{\Vl}{\operatorname{Vl}}

\newcommand{\vl}{\operatorname{vl}}

\newcommand{\E}{\mathbb E}
\newcommand{\Pbb}{\mathbb P}

\newcommand{\R}{\mathbb R}

\newcommand{\PhiStd}{\Phi}
\newcommand{\Var}{\mathbb{V}}

\newcommand{\norm}[1]{\left\|#1\right\|}
\newcommand{\Ent}{\operatorname{Ent}}
\DeclareMathOperator{\sep}{sep}

\title{Cutoff for asymmetric shelf shuffle}
\author[R. Tripathi]{Raghavendra Tripathi }
\address{Raghavendra Tripathi. Division of Science, NYU Abu Dhabi, Abu Dhabi, UAE}
\email{r.tripathi@nyu.edu}
\thanks{I thank Prof. Persi Diaconis for his encouraging comments and for suggesting several useful references. I also thank Prof. Jason Fulman for answering some questions regarding~\cite{DiaconisFulmanHolmes2013}.}
\begin{document}

\begin{abstract}
A mechanical shuffler consists of $m$ shelves. A deck of $n$ cards, arranged in increasing order, is dealt from the bottom sequentially. Each card is assigned a shelf uniformly at random and placed on the top (bottom) of the existing pile with probability $p$ ($1-p$) independently. We refer to this as asymmetric shelf-shuffle. We find the law $\nu_{n, m}^{(p)}$ of the permutation induced by the asymmetric shelf-shuffle and show that the pair consisting of the number of descents and the number of valleys is a sufficient statistic. This generalizes a result of Diaconis, Fulman, and Holmes (Ann. Appl. Prob., 2013) corresponding to the case $p=1/2$. For $p=1/2$, Chen and Ottolini (ECP, 2025) established the cutoff in the total variation distance near $\lfloor n^{5/4}\rfloor$. We establish the cutoff for the asymmetric shelf shuffle.
Let $\nu_n$ be the uniform measure on the set of all permutations $S_n$ of $\{1, \ldots, n\}$. For a fixed $p\neq 1/2$ and $c>0$, we show that 
\[\TV\left(\nu_{n, \lfloor cn^{3/2}\rfloor }^{(p)}, \nu_n\right)=1-2\Phi\left(-\frac{|2p-1|}{4\sqrt{3}c}\right)+O_{c, p}(n^{-1/2})\;.\]
We also establish the cutoff in the separation distance near $m\approx n^{2}$ and in the relative entropy near $m=n^{3/2}$. In both cases, we also obtain the cutoff profile explicitly. 


\end{abstract}

\maketitle
%%%%%%%%%%%%%%%

\section{Introduction}
\label{sec:Intro}
Card shuffling provides a rich source of problems that have real-life applications as well as deep mathematical underpinnings.  We refer the reader to the recent book~\cite{DiaconisFulman2023Shuffling} by Diaconis and Fulman for a modern and comprehensive account of various card shuffling schemes and their connections with other areas of mathematics. One of the very first questions about any shuffling scheme is how good or effective a shuffling scheme is. There are several ways to measure the quality of a shuffling scheme, but the most commonly used measure is the (total variation) distance between the law of the permutation induced by a shuffling scheme and the uniform measure on the set of all permutations of the deck. For the riffle-shuffle, Bayer and Diaconis~\cite{bayer1992trailing} famously proved that seven shuffles suffice to be sufficiently close to the uniform distribution. We refer the reader to~\cite[Table 1.2]{DiaconisFulman2023Shuffling} for a table of total variation distance after repeated riffle shuffle. Many shuffling schemes exhibit a cutoff phenomenon: the total variation distance abruptly decreases from $\approx 1$ to $\approx 0$ in a short window. The term cutoff was introduced by Aldous and Diaconis in the context of card shuffling~\cite{aldous1986shuffling}. For the asymmetric riffle shuffle, the cutoff phenomenon was recently established by Sellke~\cite{sellke2022cutoff}. The cutoff phenomenon for Markov chains has now become an active area of research in its own right~\cite{salez2025modern}. 

In this paper, we prove the cutoff phenomenon and obtain the cutoff profile for asymmetric shelf-shuffles. We begin with the definition of asymmetric shelf shuffle introduced in~\cite{kuba}. 
\begin{definition}[Asymmetric shelf-shuffle]
Let $m\in \mathbb{N}$ be a natural number and $p\in (0, 1)$. An $(m, p)$-shelf shuffler has $m$ labelled shelves. A deck of cards labelled $1, \ldots, n$ is dealt from the bottom sequentially, and it is assigned a shelf independently and uniformly at random. Each card is placed at the top or bottom of the existing pile, on the assigned shelf, with probability $p$ and $1-p$, respectively, independently of everything else. The shuffled deck is then formed by joining the piles in the sequential order. 
\end{definition}
When $p=1/2$, we refer to this as the \emph{symmetric shelf-shuffle} or simply as \emph{shelf-shuffle}. The symmetric shelf-shuffling schemes are often used by the mechanical shufflers in casinos (with $m\approx 10$). A casino company using the shelf shuffler was interested in understanding the performance of the shelf shuffler after \emph{a single shuffle}. The company hired Diaconis, Fulman, and Holmes to study this shuffling scheme, who were the first to mathematically analyze the shelf shuffle~\cite{DiaconisFulmanHolmes2013}. Among several other things, they proved that the number of valleys is a sufficient statistic for the symmetric shelf-shuffle. Let $\nu_{n, m}\equiv \nu_{n, m}^{(p)}$ denote the law of the permutation induced by $(m, p)$-shelf shuffler with $n$ cards. Diaconis, Fulman, and Holmes~\cite{DiaconisFulmanHolmes2013} proved
 \begin{equation}
 \label{eqn:LawSymShuffle}
\nu_{n, m}^{(1/2)}(\pi) = \frac{4^{v(\pi)+1}}{2(2m)^n} \sum_{a=0}^{m-1}\binom{n+m-a-1}{n}\binom{n-1-2v(\pi)}{a-v(\pi)},
\end{equation}
where $v(\pi)=|\{1<i<n: \pi_{i-1}>\pi_i<\pi_{i+1}\}|$  is the number of valleys in permutation $\pi$. Chen and Ottolini~\cite{chen2025cutoff} established the cutoff phenomenon for the symmetric shelf shuffle, answering a question in~\cite[Problem 2, Section 11.5]{DiaconisFulman2023Shuffling}. More precisely, let $\nu_n$ be the uniform measure on $S_n$. Chen and Ottolini~\cite{chen2025cutoff} proved that for any $c>0$
\begin{equation}
\label{eqn:SymCutoff}
\TV(\nu_{n, \lfloor cn^{5/4}\rfloor}, \nu_n) = 1-2\Phi\round{-\frac{1}{12c^2\sqrt{10}}}+O_c\round{\frac{1}{\sqrt{n}}},
\end{equation}
where $\Phi(t) = \frac{1}{\sqrt{2\pi}}\int_{-\infty}^{t}e^{-x^2/2}\,dx$ is the cumulative distribution function of a standard Gaussian, and $\TV(\cdot, \cdot)$ denotes the total variation distance between two probability measures. Here and throughout this paper, we use $f_n=O(g_n)$ to denote that $|f_n|\leq C|g_n|$ for some constant $C>0$. We use the subscript {$O_{\alpha_1,\ldots,\alpha_k}$} to emphasize the dependence of the constant $C$ on {parameters $\alpha_1,\ldots,\alpha_k$}. 

Very recently, Clay~\cite{clay2025limit} studied the descents and inversions in the symmetric shelf-shuffle. Card guessing games for the symmetric shelf shuffle with a single shelf were investigated in~\cite{clay2025guessing, tripathi2026position}. The asymmetric shelf-shuffle was introduced in~\cite{kuba}, where results like CLT, local CLT, and the LDP for the number of correct guesses in complete feedback games for $(1, p)$-shelf shuffles are proved. Despite this recent attention, several natural questions about the shelf-shuffle remain open and offer a promising avenue for future research.

In this paper, we investigate the mixing properties of the asymmetric shelf shuffle.  We determine the law of the permutation induced by an asymmetric shelf shuffle and deduce that the pair $(\text{number of valleys}, \text{number of descents})$ is a sufficient statistic for the asymmetric shelf shuffle. In the special case when $p=1/2$, {our formula recovers}~\cite[Theorem 3.1]{DiaconisFulmanHolmes2013}.

Using this, we prove a cutoff phenomenon (with respect to the number of shelves) for the asymmetric shelf-shuffle. Somewhat surprisingly, the cutoff location in this case is near $m=\lfloor n^{3/2}\rfloor$ as opposed to $\lfloor n^{5/4}\rfloor$ in the symmetric case.
Following Chen and Ottolini~\cite{chen2025cutoff}, we show that the permutation induced by an asymmetric shelf-shuffle is almost an exponential tilt of the uniform measure on $S_n$. When $p\neq 1/2$, the dominant term in the exponential tilt is due to the number of descents, and there is a lower-order term due to the number of valleys. The tilt due to the number of descents vanishes when $p=1/2$, and the only contribution comes from the number of valleys, as it was observed in~\cite {chen2025cutoff}. In an upcoming paper by the author, we study the statistics of the asymmetric shelf shuffle. 

\subsection{Setup and main results}
Before we state our results, we fix some necessary notations. Let $\nu_{n, m}^{(p)}$ denote the law of permutation induced by an $(m, p)$-shelf shuffle on a deck of $n$ cards. Throughout our discussion, we will assume that $p\in (0, 1)$ and $q:=1-p$. For a permutation $\pi=(\pi_1, \ldots, \pi_n)\in S_n$, we define the descent set and the valley set, respectively, as
\[
\Des(\pi) = \{i\in [n-1]: \pi_{i}>\pi_{i+1}\}, \quad \Vl(\pi) = \{2\leq i\leq n-1: \pi_{i-1}>\pi_i<\pi_{i+1}\}\;.
\]
We also define the ascent set $\Asc(\pi) := [n-1]\setminus\Des(\pi)$. We write $\des(\pi) = |\Des(\pi)|$, $\asc(\pi) = |{\Asc}(\pi)|$, and $\vl(\pi) = |\Vl(\pi)|$ for the number of descents, ascents, and valleys in a permutation $\pi$, respectively. Our first result determines the law of the permutation induced by an asymmetric shelf shuffle.

\begin{theorem}
\label{thm:Law}
For $\pi\in S_n$, let 
\[
\mathcal{B}(\pi):=\{B\subseteq [n-1]: \forall j \in \Vl(\pi), \{j-1, j\}\not\subseteq B\}. 
\]
Then, for any $\pi\in S_n$ we have
\begin{equation}
\label{eqn:LawAsym}
\nu_{n, m}^{(p)}(\pi) = \frac{1}{m^n}\sum_{B\in \mathcal{B}(\pi)}\binom{m}{{n-|B|}}p^{|\Asc(\pi)\cap B|}q^{|\Des(\pi)\cap B|}\;.
\end{equation}
\end{theorem}
For a permutation $\pi\in S_n$, we define a polynomial 
\[
F_{\pi}(u) := \sum_{B\in \mathcal{B}(\pi)}u^{|B|}p^{|\Asc(\pi)\cap B|}q^{|\Des(\pi)\cap B|}\;.
\]
It follows from~\eqref{eqn:LawAsym} that 
\begin{equation}
\label{eqn:LawPoly}
\nu_{n, m}^{(p)}(\pi) = \frac{1}{m^n}\sum_{j=0}^{n-1}\binom{m}{n-j}[u^j]F_{\pi}(u)=\frac{1}{m^n}[u^n](1+u)^mF_{\pi}(u)\;,
\end{equation}
where we use the notation $[u^j]P(u)$ to denote the coefficient of $u^j$ in the polynomial $P$. Our next proposition shows that the polynomial $F_{\pi}(u)$ depends only on $\des(\pi), \asc(\pi),$ and $\vl(\pi)$. Since $\des(\pi)+\asc(\pi)=n-1$, it follows that $F_{\pi}(u)$ is completely determined by the pair $(\des(\pi), \vl(\pi))$ and consequently $(\des(\pi), \vl(\pi))$ \emph{is a sufficient statistic} for $\nu_{n, m}^{(p)}$.
\begin{proposition}
\label{prop:F_pi}
Let $\pi\in S_n$ such that $d=\des(\pi), a=\asc(\pi)$ and $v=\vl(\pi)$. Let $F_{\pi}$ be as defined above.  Then, 
\[F_\pi(u) = (1+pu)^{a-v}(1+qu)^{d-v}(1+u)^v\;.\]
%In particular, $(\des(\pi), \vl(\pi))$ is a sufficient statistics for the asymmetric shelf shuffle. 
\end{proposition}
As a corollary of Proposition~\ref{prop:F_pi} and~\eqref{eqn:LawPoly}, we obtain an alternate formula for $\nu_{n, m}^{(p)}(\pi)$.
\begin{corollary}
\label{cor:LawAsym_closedform}
Let $\pi\in S_n$ such that $d=\des(\pi), a=\asc(\pi)$, and $v=\vl(\pi)$. Then, 
\begin{align*}
\nu_{n, m}^{(p)}(\pi) &= \frac{1}{m^n}[u^{n}](1+u)^{m+v}(1+pu)^{a-v}(1+qu)^{d-v}\\
&= \frac{1}{m^n}\sum_{r=0}^{a-v}\sum_{s=0}^{d-v}\binom{a-v}{r}\binom{d-v}{s}p^rq^s\binom{m+v}{n-r-s}\;.
\end{align*}
%\[
%\nu_{n, m}^{(p)}(\pi) = \frac{1}{m^n}\sum_{j=0}^{n-1}\binom{m}{n-j}\sum_{r, s}p^{a-v-r}q^{d-v-s}\binom{a-v}{r}\binom{d-v}{s}\binom{v}{j-r-s}\;.
%\]
In particular, $(\des(\pi), \vl(\pi))$ is a sufficient statistic for the asymmetric shelf shuffle. 
\end{corollary}
In Section~\ref{subsec:Symmetric}, we specialize Corollary~\ref{cor:LawAsym_closedform} to $p=1/2$ to obtain the law of symmetric shelf-shuffle in Diaconis-Fulman-Holmes~\cite[Theorem 3.1]{DiaconisFulmanHolmes2013}. 

Our next result establishes the cutoff for the asymmetric shelf shuffle.
\begin{theorem}
\label{thm:cutoff}
Let $\nu_n$ be the uniform measure on $S_n$.
Fix $p\in (0, 1)\setminus\{1/2\}$. Let $c>0$ be fixed and $m_n=\lfloor cn^{3/2}\rfloor$. Then, 
\[
\TV(\nu_{n, m_n}^{(p)}, \nu_n) = 1- 2\Phi\round{-\frac{|2p-1|}{4\sqrt{3}c}}+O_{c, p}(n^{-1/2})\;.
\]
\end{theorem}

While total variation distance is arguably the most widely studied metric between probability measures, several other measures of distance are natural and have been studied. Let $\mu$ be a probability measure on $S_n$ and $\nu_n$ be the uniform measure on $S_n$. We define the \emph{relative entropy} of $\mu$ as 
\[
\Ent(\mu):=\Ent(\mu\mid \nu_n) := \sum_{\pi\in S_n}\mu(\pi)\log \frac{\mu}{\nu_n}(\pi)\;. 
\]
While relative entropy is not a metric, it is a good measure of how different $\mu$ is from the uniform measure. Our next result gives the cutoff profile for the asymmetric shelf-shuffle in relative entropy.
\begin{theorem}
\label{thm:entropy-cutoff}
Fix {$p\in(0,1)\setminus\{1/2\}$} and {let $q=1-p$}. Let $c>0$ be fixed and let $m=m_n=\lfloor c n^{3/2}\rfloor$.
% \[
% m=m_n=\lfloor c n^{3/2}\rfloor,\qquad c>0.
% \]
Then
\[
\lim_{n\to \infty}\Ent\!\left(\nu^{(p)}_{n,m}\right)
=
\frac{(2p-1)^2}{24c^2}\;.
\]
\end{theorem}

Let us also define the \emph{separation} as 
\[
\sep(\mu) = \max_{\pi\in S_n}\left(1-\frac{\mu(\pi)}{\nu_n(\pi)}\right)\;.
\]
The following theorem gives the limit profile in the separation distance. 
\begin{theorem}
\label{thm:sep-cutoff}
Fix $p\in(0,1)\setminus\{1/2\}$. Fix $c>0$ and let $m_n=\lfloor c n^2\rfloor$. Then, as $n\to \infty$, we have
\[
\operatorname{sep}\!\left(\nu^{(p)}_{n,m_n},\nu_n\right)
\longrightarrow
1-\exp\left\{-\frac{|2p-1|}{2c}\right\}.
\]
\end{theorem}

Recall that the casino implementing the shelf shuffle was interested in the deviation of the single shuffle using the shelf shuffle from the uniform distribution. These cutoff results precisely address this question. In the next section, we discuss the question of repeated shuffle.

\subsection{Asymmetric shelf-shuffle under repetition}
\label{subsec:Repeatation}
 
From a practical viewpoint, it is arguably more natural to ask how many $(m, p)$-shelf shuffles it takes for a deck of $n$ cards to be sufficiently close to the uniform distribution, where $m\geq 1$ and $p\in (0, 1)$ are fixed. When $p=1/2$, Diaconis, Fulman, and Holmes~\cite{DiaconisFulmanHolmes2013} proved a composition rule showing that an $(m_1, 1/2)$-shelf shuffle followed by an $(m_2, 1/2)$-shelf shuffle is equivalent to a $(2m_1m_2, 1/2)$-shelf shuffle. In particular, repeating a symmetric shelf shuffle with a single shelf $k$ times is equivalent to a symmetric shelf shuffle with $2^{k-1}$ shelves. Chen and Ottolini~\cite{chen2025cutoff} exploited this along with their cutoff result~\eqref{eqn:SymCutoff} to conclude that it takes $\frac{5}{4}\log_{2}n$ shuffles for a symmetric shelf shuffle with a single shelf to mix. 

It is natural to ask the same question for the asymmetric shelf shuffle with a single shelf. 
\begin{question}\label{ques:cutoff}
How many times does one need to shuffle a deck of $n$ cards using {a $(1,p)$-shelf shuffle} for it to mix? 
\end{question}
For the asymmetric shelf shuffle, the composition rule fails, as we show below. In other words, the asymmetric shelf shuffle is not a closed family under repetition. In particular, Theorem~\ref{thm:cutoff} does not directly yield an answer to Question~\ref{ques:cutoff}. We leave it to the future to understand the composition rule for $(m_1, p_1)$-shelf shuffle followed by $(m_2, p_2)$-shelf shuffle and thereby answer Question~\ref{ques:cutoff}. 
\begin{proposition}
The asymmetric shelf shuffle is not closed under repeated shuffling. 
\end{proposition}
\begin{proof}
We will show this via a simple counterexample. Let $n=3$. recall that for any $p\in (0, 1)$, $\nu_{3, m}^{(p)}$ denotes the law of $(m, p)$-shelf shuffle with with $n=3$ cards. For any $s\in (0, 1)$, Using Corollary~\ref{cor:LawAsym_closedform}, we obtain
\begin{align*}
{B_{m, s}} &{:=\nu_{3, m}^{(s)}(132) = \nu_{3, m}^{(s)}(231) =  \frac{m^2-1}{6m^2}+ \frac{s(1-s)}{m^2},}\\
{C_{m, s}} &{:= \nu_{3, m}^{(s)}(213) = \nu_{3, m}^{(s)}(312)= \frac{m^2-1}{6m^2},}\\
{D_{m, s}} &{:= \nu_{3, m}^{(s)}(321) = \frac{(m-1)(m-2)+6(m-1)(1-s)+6(1-s)^2}{6m^2}\;.}
\end{align*}
Note that we write the permutations in one-line notation. On the other hand, let $\mu:=(\nu_{3, 1}^{(p)})^{\ast 2}$ denote the convolution of $\nu_{3, 1}^{(p)}$ with itself, and thus, the law of the permutation induced by repeating twice the $(1, p)$-shelf shuffling. By a direct computation, we obtain 
\begin{align*}
{B_p'} &{:=\mu(132)= \mu(231) = p-3p^2+5p^3-3p^4,}\\
{C_p'}&{:= \mu(213)=\mu(312) = p(1-p)^2,} \\
{D_p'} &{:=\mu(321) = 3p^2(1-p)^2\;.}
\end{align*}
Comparing $C_{m, s}$ with $C_p'$, one can already see that in general, there are no $(m, s)$ such that {$\mu=\nu_{3, m}^{(s)}$} unless $p=1/2$. In fact, it is an easy verification to show that {$B_{m, s}, C_{m,s}, D_{m, s}$} satisfy 
\[
{(B_{m, s}+C_{m, s}+D_{m, s})^2 - D_{m, s}-\frac{1}{2}C_{m, s}=0}\;.
\]
On the other hand, one can check that 
\[{(B_p'+C_p'+D_p')^2 - D_p'-\frac{1}{2}C_p' = \frac{1}{2}p(1-p)^2(2p-1)}\;. \]
For $\mu$ to be realizable by an {$(m, s)$-shelf shuffle}, the left-hand side must vanish, and hence $p=1/2$ if $p\in (0, 1)$. 
\end{proof}

\subsection{Symmetric shelf shuffle}
\label{subsec:Symmetric}
Note that for $p=q=1/2$, we get $F_{\pi}(u) = (1+u/2)^{n-1-2v}(1+u)^v$.  In particular, $F_{\pi}(u)$ only depends on the number of valleys $\vl(\pi)$ in $\pi$. Therefore, $\vl(\pi)$ is a sufficient statistic for the symmetric shelf shuffle. We end this section by showing how one recovers~\eqref{eqn:LawSymShuffle} for the symmetric shelf shuffle. 
%It is immediate that the coefficient of $u^j$ in $F_{\pi}(u)$ is $\sum_{k=0}^{j}\binom{v}{j-k}\binom{n-1-2v}{k}2^{-k}$. Therefore, 
%\[
%\nu_{n, m}^{(1/2)}(\pi) = \frac{1}{m^n}\sum_{j=0}^{n-1}\binom{m}{n-j}\sum_{k=0}^{j}\binom{v}{j-k}\binom{n-1-2v}{k}2^{-k}
%\]
%Using Chu--Vandermond identity and summing over $j$ first we obtain 
%\[
%\nu_{n, m}^{(1/2)}(\pi) = \frac{1}{m^n}\sum_{k=0}^{n-1-2v}\binom{m+v}{n-k}\binom{n-1-2v}{k}2^{-k} = \frac{1}{m^n} [u^n] (1+u)^{m+v}(1+u/2)^{n-1-2v}\;.
%\]
Using $a+d=n-1$ and $p=q=1/2$ in Corollary~\ref{cor:LawAsym_closedform}, we know that
\[
\nu_{n, m}^{(1/2)}(\pi)= \frac{1}{m^n} [u^n] (1+u)^{m+v}(1+u/2)^{n-1-2v}.
\]
{Now set} $N=n-1-2v$, and rewrite $(1+u)^{m+v}(1+u/2)^{n-1-2v}$  as
\begin{align*}
 \frac{1}{2^{N}} (1+u)^{m+v+N} \round{1+ \frac{1}{1+u}}^{N}&=  \frac{1}{2^{N}} \sum_{r=0}^{N}\binom{N}{r}(1+u)^{N+m+v-r}\;.
\end{align*}
Extracting the coefficient of $u^n$ in $(1+u)^{N+m+v-r}$, we obtain
\[
\nu_{n, m}^{(1/2)}(\pi)=\frac{2\cdot 4^{v} }{(2m)^n}\sum_{r=0}^{n-1-2v}\binom{n-1-2v}{r}\binom{n+m-v-r-1}{n}\,
\]
which is equivalent to~\eqref{eqn:LawSymShuffle} up to a simple change of variable {$a=r+v$}. We thus recover the distribution of symmetric shelf shuffle given in~\cite[Theorem 3.1]{DiaconisFulmanHolmes2013}.

%%%%%%%%%%%%%%%%%%%%%%%%%%%%%%%%%%%%%

\section{Proofs}\label{sec:Proofs}
We will first prove some preliminary results, using which we conclude the proof of Theorem~\ref{thm:Law} and Proposition~\ref{prop:F_pi}. Finally, we prove Theorem~\ref{thm:cutoff} assuming Proposition~\ref{prop:descent_tilt_measure} and Proposition~\ref{prop:GaussianProfile}. The proof of these two propositions requires more preparation. We defer these to later sections.

\subsection{Proofs of Theorem~\ref{thm:Law} and Proposition~\ref{prop:F_pi}}

Suppose that a set of cards $c_1<c_2<\cdots<c_k$ is assigned to a shelf. Since the cards are drawn from the deck in decreasing order, the card $c_k$ must be the first card assigned to this shelf. To every other card $\{c_j: 1\leq j\leq k-1\}$, we assign an independent Bernoulli $p$ random variable $X_{c_j}$. If $X_{c_j}=1$, we place $c_j$ at the top of the existing pile; if $X_{c_j}=0$, we place it at the bottom of the pile. Any realization of these Bernoullis yields a permutation of $\{c_1, \ldots, c_k\}$ \emph{which has no valleys}. Conversely, any permutation of $\{c_1, \ldots, c_k\}$ which contains no valleys is attained by a unique realization of {Bernoulli variables} $\{X_{c_j}: 1\leq j\leq k-1\}$. 
\begin{lemma}
\label{lem:AscDes}
Fix a shelf containing the cards $c_1<c_2<\cdots<c_k$. Fix a realization of i.i.d. Bernoulli (top/down) random variables  $\{X_{c_j}: 1\leq j\leq k-1\}$ as above and let $\pi=(c_{\pi_1}, c_{\pi_2}, \ldots, c_{\pi_k})$ be the corresponding permutation. Then, 
\[
|\Asc(\pi)| = \sum_{i=1}^{k-1}X_{c_i}, \qquad |\Des(\pi)| = \sum_{i=1}^{k-1}(1-X_{c_i})\;.
\]
\end{lemma}
\begin{proof}
%Since $|\Asc(\pi)|+|\Des(\pi)|=k-1$, it suffices to prove only the first claim. To this end, we 
Note that if $X_{c_j}=1$, then the card $c_j$ is placed on the top of the existing pile. Since the cards are dealt in decreasing order, it follows that the card $c_{j'}$ following $c_j$ must satisfy $c_{j'}>c_{j}$. Therefore, there is an ascent at $c_j$. 
On the other hand, if $X_{c_j}=0$, then the card $c_j$ is placed at the bottom of the existing pile. Since the cards are dealt in decreasing order, the card preceding $c_j$ must be larger than $c_j$. In other words, there is a descent at the card preceding $c_{j}$ in the final deck. Since $\Asc(\pi)$ and $\Des(\pi)$ are disjoint sets and $|\Asc(\pi)|+|\Des(\pi)|=k-1$, {the claim follows}.

\end{proof}

Let $\pi=(\pi_1, \pi_2, \ldots, \pi_n)\in S_n$ be a permutation. Fix $m\geq 1$. Note that $\pi$ is realized in an $m$-shelf shuffle if and only if there exists $0\leq b\leq m-1$ (where $b$ denotes the number of empty shelves) such that $\pi$ can be partitioned into $t=m-b$ non-empty parts such that if we partition $\pi$ into $t$ non-empty blocks
\[
\pi_1, \ldots, \pi_{i_1}\vert \pi_{i_1+1}\ldots \pi_{i_2}\vert \cdots {\vert\, \pi_{i_{t-1}+1}\ldots \pi_n}\,.
\]
Then each is a valley-free permutation. Furthermore, if we fix the $b$ shelves that are empty, then there is a unique way to place $t$ parts onto $t$ remaining shelves that gives rise to the permutation $\pi$. Given a $\pi\in S_n$, let us call a partition of $\pi$ an \emph{admissible partition} if each part is valley-free. We first claim that for a permutation $\pi$, the admissible partitions are in one-to-one correspondence with $\mathcal{B}(\pi)$.

\begin{lemma}
\label{lem:BlockStructure}
Let $\pi\in S_n$. The set of admissible partitions of $\pi$ is in one-to-one correspondence with $\mathcal{B}(\pi)$.
\end{lemma}
\begin{proof}
Let $\pi\in S_n$. Suppose we are given an admissible partition of $\pi$ as follows: 
\[
\pi_1, \ldots, \pi_{i_1}\vert \pi_{i_1+1}\ldots \pi_{i_2}\vert \cdots {\vert\, \pi_{i_{t-1}+1}\ldots \pi_n}\,.
\]
For each $i\in [n-1]$, declare $i\in B$ if $\pi_{i}$ and $\pi_{i+1}$ both lie in the same part (that is, both cards appear on the same shelf). Therefore, {$B=[n-1]$ if $t=1$}, and $B=[n-1]\setminus\{i_1, i_2, \ldots, i_{t-1}\}$ if $t\geq 2$. We claim that $B\in \mathcal{B}(\pi)$. Indeed, if $\alpha$ is a valley, then $\pi_{\alpha-1}>\pi_{\alpha}<\pi_{\alpha+1}$. Since every part in an admissible partition is valley-free, the three consecutive entries $\pi_{\alpha-1},\pi_{\alpha},\pi_{\alpha+1}$ cannot all lie in the same part. Equivalently, at least one of the adjacent indices $\alpha-1$ and $\alpha$ is a cut, so $\{\alpha-1,\alpha\}\not\subseteq B$. Hence $B\in\mathcal B(\pi)$.
It is clear that distinct partitions of $\pi$ give rise to distinct $B$, that is, this map is one-to-one. 

Conversely, fix $B\in \mathcal{B}(\pi)$. If $B=[n-1]$, then $\pi$ cannot have any valleys. In this case, set the partition of $\pi$ to be {a single block}, which is clearly admissible. Now suppose $I:=[n-1]\setminus B =\{i_1<i_2<\ldots<i_{t-1}\}$ for some $t\geq 2$. Define a partition of $\pi$ as:
\[
\pi = \pi_1\ldots \pi_{i_1}\vert \pi_{i_1+1}\ldots \pi_{i_2}\vert \cdots\cdots \vert \pi_{i_{t-1}+1}\ldots \pi_n\;.
\]
We claim that the above partition is admissible. To see this, let $\alpha$ be a valley in $\pi$. Since $B\in \mathcal{B}(\pi)$, we must have $\{\alpha-1, \alpha\}\not\subseteq B$. Therefore, we must have that $I\cap \{\alpha, \alpha-1\}\neq\emptyset$. Therefore, a block must {end at} either $\alpha-1$ or $\alpha$. In other words, $\{\alpha-1, \alpha, \alpha+1\}$ cannot be in the same block of the partition. Therefore, every block must be valley-free. 

\end{proof}

\begin{proof}[Proof of Theorem~\ref{thm:Law}]
Fix $\pi\in S_n$. Let $B\in \mathcal{B}(\pi)$ and consider the admissible partition of $\pi$ into $t=n-|B|$ non-empty blocks as $C_1, C_2, \ldots, C_t$ in Lemma~\ref{lem:BlockStructure}. Since each card is assigned to a shelf independently and uniformly, the probability of this happening is $1/m^{n}$. Furthermore, to realize these blocks as non-empty shelves, we must choose $t$ shelves out of $m$ shelves, and assign block $C_{\ell}$ to the $\ell$-th chosen shelf, which can be done in $\binom{m}{n-|B|}$ ways. Once the blocks are assigned to shelves, the decision to place the card on top/bottom uniquely determines a realization of top/down i.i.d. Bernoulli random variables. By Lemma~\ref{lem:AscDes}, the probability of this realization is $p^{|\Asc(\pi)\cap B|}q^{|\Des(\pi)\cap B|}${.}
Therefore, we obtain 
\[
\nu_{n, m}^{(p)}(\pi)=\frac{1}{m^n}\sum_{B\in \mathcal{B}(\pi)}\binom{m}{n-|B|}p^{|{\Asc}(\pi)\cap B|}q^{|\Des(\pi)\cap B|}\;.
\]
\end{proof}

\begin{proof}[Proof of Proposition~\ref{prop:F_pi}]
{Fix $\pi\in S_n$. Let}
\[
{V^{+}(\pi):=\Vl(\pi), \qquad V^{-}(\pi):=\{j-1: j\in \Vl(\pi)\}.}
\]
{Then $V^{+}(\pi)\subseteq \Asc(\pi)$, $V^{-}(\pi)\subseteq \Des(\pi)$, and $|V^{+}(\pi)|=|V^{-}(\pi)|=\vl(\pi)$. Moreover, the pairs $\{j-1,j\}$, $j\in\Vl(\pi)$, are disjoint. Put $A_0:=\Asc(\pi)\setminus V^{+}(\pi)$ and $D_0:=\Des(\pi)\setminus V^{-}(\pi)$.}

{For each $i\in A_0$, the set $B$ either omits $i$ or includes $i$, contributing a factor $1+pu$. Similarly, each $i\in D_0$ contributes a factor $1+qu$. For a valley $j\in\Vl(\pi)$, the admissibility condition allows $B$ to contain neither $j-1$ nor $j$, to contain only $j-1$, or to contain only $j$, but not both. The corresponding contribution is}
\[
{1+qu+pu=1+u.}
\]
{Multiplying these independent contributions gives}
\[
{F_\pi(u)=(1+pu)^{|A_0|}(1+qu)^{|D_0|}(1+u)^{\vl(\pi)}=(1+pu)^{a-v}(1+qu)^{d-v}(1+u)^v,}
\]
{as claimed.}
\end{proof}

%%%%%%%%%%%%%%%%%%%%%%%%%%%%%
%%%%%%%%%%%%%%%%%%%%%%%%%%%%%
\subsection{Proof of Theorem~\ref{thm:cutoff}}
The proof idea is inspired by~\cite {chen2025cutoff}. Let $\nu_n$ denote the uniform measure on $S_n$. We first show that $\nu_{n, m}^{(p)}$ is an exponential tilt of $\nu_n$, where the tilt depends on the number of descents as well as the number of valleys. To state this result more precisely, we need some notations. From now on, we set 
\[m\equiv m_n=\lfloor cn^{3/2}\rfloor, \quad M_n=m-n+1, \quad \tau_n={\frac{n}{M_n}}= \frac{1}{c\sqrt{n}}+ O_c(n^{-1})\;.\]
Recall that 
\[F_{\pi}(u) = (1+pu)^{\asc(\pi)-\vl(\pi)}(1+qu)^{\des(\pi)-\vl(\pi)}(1+u)^{\vl(\pi)}\;.\]
Given $\pi\in S_n$, let $J_{\pi}$ be a random variable with probability generating function
\[
\mathbb{E}_{\pi}[z^{J_{\pi}}] = \frac{F_{\pi}(\tau_nz)}{F_{\pi}(\tau_n)}\;.
\]
Here and throughout, we use the subscript $\mathbb{E}_{\pi}$ and $\Var_{\pi}$ to denote the expectation and variance with respect to the law of $J_{\pi}$. 
Let us also define 
\begin{align}
%\label{eqn:thetaAndeta}
     \theta_n &:= \log\frac{1+q\tau_n}{1+p\tau_n}=-\frac{2p-1}{c\sqrt{n}}+ O(n^{-1}),\label{eqn:Thetan} \\
     \eta_n &:=  \log\frac{1+\tau_n}{(1+p\tau_n)(1+q\tau_n)}=-\frac{pq}{c^2n}+O(n^{-3/2})\;.\label{eqn:Etan}
\end{align}
The asymptotics of $\theta_n$ and $\eta_n$ in~\eqref{eqn:Thetan} and~\eqref{eqn:Etan} follow from a simple Taylor expansion. We skip the details.

\begin{lemma}[$\nu_{n, m_n}^{(p)}$ is an exponential tilt of $\nu_n$]
Let $\nu_{n, m}^{(p)}$ be as above. Then, 
\[
L_n(\pi)= \frac{\nu_{n, m_n}^{(p)}{(\pi)}}{\nu_n{(\pi)}} = C_n \exp\left(\theta_n \des(\pi)+ \eta_n\vl(\pi)\right)\mathbb{E}{_{\pi}}[R_{n, J_{\pi}}],
\]
where  
\[
C_n =n!\binom{m}{n} \frac{(1+p\tau_n)^{n-1}}{m^n}, \qquad R_{n, j}:=\prod_{\ell=0}^{j-1}\frac{1-\ell/n}{1+\ell/M_n}\quad \forall j\geq 1\;.
\]
\end{lemma}
\begin{proof}
Recall that 
\[
L_n(\pi) = n!\nu_{n, m}^{(p)}(\pi) = \frac{n!}{m^n}\sum_{j=0}^{n-1}\binom{m}{n-j}[u^j]F_{\pi}(u)\;.
\]
Using the fact that $\binom{m}{n-j} = \binom{m}{n}\tau_n^j R_{n, j},$  summing over $j$, and using the definition of $J_{\pi}$, we get 
\[
L_n(\pi) = \frac{n!}{m^n}\binom{m}{n}F_{\pi}(\tau_n)\mathbb{E}{_{\pi}}[R_{n, J_{\pi}}]\;.
\]
The proof now finishes by observing that $F_{\pi}(\tau_n) = (1+p\tau_n)^{n-1}\exp(\theta_n\des(\pi)+\eta_n\vl(\pi))$.

\end{proof} 

%More precisely, 
%\[
%L_n(\pi):= \frac{\nu_{n, m_n}^{(p)}}{\nu_n} = n!\nu_{n, m}^{(p)} = C_n' \exp\left(\theta_n \des(\pi)+ \eta_n\vl(\pi)\right)\;,
%\]
%Note that $\theta_n= \frac{2p-1}{c\sqrt{n}}+ O(n^{-1})$ and $\eta_n=-\frac{pq}{c^2n}+O(n^{-3/2})$. 
To compare it with~\cite{chen2025cutoff}, note that when $p=1/2$, the dominating contribution comes from the tilt due to valleys. This is precisely what was exploited in~\cite{chen2025cutoff}. However, our next {proposition} shows that when $p\neq 1/2$, the dominating term in the exponential tilt is due to the number of descents. The contribution from the number of valleys is, in fact, negligible. This is our next {proposition, whose proof} we defer to Section~\ref{sec:FirstMainLemma}.
\begin{proposition}
\label{prop:descent_tilt_measure}
Let $\mu_n^{\theta_n}$ be a measure on $S_{n}$ {with} the Radon-Nikodym derivative
\[
L_n^{\theta_n}:=\frac{\mu_n^{\theta_n}(\pi)}{\nu_n(\pi)} =  \frac{\exp(\theta_n \des(\pi))}{\mathbb{E}_{\nu_n}[\exp(\theta_n \des(\pi))]}.
\]
Then, 
\[\TV({\nu_{n, m_n}^{(p)}}, \mu_n^{\theta_n})= O{_{c,p}}(n^{-1/2})\;.\]
\end{proposition}
Because of Proposition~\ref{prop:descent_tilt_measure}, it suffices to obtain the cutoff profile for $\mu_n^{\theta_n}$. In the setup of Chen and Ottolini~\cite{chen2025cutoff}, the measure $\mu_{n, m}^{(1/2)}$ is already an exponential tilt of the uniform measure, although the tilt depends on the number of valleys instead of the number of descents. This is the main difference, because of which the technical details in our case are different and need to be worked out independently. Our next {proposition} shows that $\mu_n^{\theta_n}$ has the correct limit profile.

\begin{proposition}
\label{prop:GaussianProfile}
 \[\TV(\mu_n^{\theta_n}, \nu_n) = 1-2\Phi\left(-\frac{|2p-1|}{4\sqrt{3}c}\right)+O_{c, p}(n^{-1/2})\;.\]
\end{proposition}
We defer the proof of Proposition~\ref {prop:GaussianProfile} to Section~\ref{sec:SecondMainLemma}. We end this section with the proof of Theorem~\ref{thm:cutoff}. 
\begin{proof}[Proof of Theorem~\ref{thm:cutoff}]
The proof follows from the triangle inequality. Note that 
\[
\TV(\nu_{n, m}^{(p)}, \mu_n^{\theta_n})- \TV(\nu_n, \mu_n^{\theta_n})\leq \TV(\nu_{n, m}^{(p)}, \nu_n) \leq \TV(\nu_{n, m}^{(p)}, \mu_n^{\theta_n})+ \TV(\nu_n, \mu_n^{\theta_n})\;.
\]
By Proposition~\ref{prop:descent_tilt_measure}, we have {$\TV(\nu_{n,m}^{(p)}, \mu_n^{\theta_n})=O_{c, p}(n^{-1/2})$}; and by Proposition~\ref{prop:GaussianProfile}, we have $\TV(\mu_n^{\theta_n}, \nu_n) = 1-2\Phi\left(-\frac{|2p-1|}{4\sqrt{3}c}\right)+O_{c, p}(n^{-1/2})$. We conclude that 
\[
\TV(\nu_{n, m}^{(p)}, \nu_n) = 1-2\Phi\left(-\frac{|2p-1|}{4\sqrt{3}c}\right)+O_{c, p}(n^{-1/2})\;.
\]
\end{proof}

%%%%%%%%%%%%%%%%%%%%%%%%%%%%%%%%%%%%%%%%%%%%%%%%%%.       Preliminaries
 %%%%%%%%%%%%
 %%%%%%%%%%%%%%%%%%%%%%%%%%%%%%%%%%%%%%
\subsection{Remaining proofs}
\label{sec:preliminaries}
We start by collecting some preliminaries about the descents and valleys in a random permutation that will be useful later. 
\subsubsection{Descents and Valleys in random permutations}
We recall some results on the moments of the number of descents $\des(\pi)$ and the number of valleys $\vl(\pi)$ in a uniform random permutation that will be useful later. 

Let $D_n:=\des(\pi)$ be the number of descents in a random permutation $\pi\sim \nu_n$. The asymptotic properties of $D_n$ are studied by several authors. We refer the interested reader to~\cite{FKLP, Pitman97} and the references therein for more details. For our purposes, we need the asymptotics for the expectation and variance of $D_n$, which is easy to derive (also see~\cite{Alperen}). 
\begin{lemma}\label{lem:des-moments}
Let $D_n:=\des(\pi)$ be the number of descents in a random permutation $\pi\sim \nu_n$. Then, 
\[
\bar d_n=\E_{\nu_n}[D_n]=\frac{n-1}{2},
\qquad
\sigma_n^2=\Var_{\nu_n}(D_n)=\frac{n+1}{12}.
\]
\end{lemma}
It is well-known that the moment generating function $\mathbb{E}[t^{D_n}]$ of $D_n$ is given by the Eulerian polynomial, which is real-rooted~\cite{Pitman97}. It follows that $D_n$ is a sum of independent Bernoulli random variables. In particular, there exist numbers $\alpha_{n,1},\dots,\alpha_{n,n-1}\in(0,1)$ such that
\[
\E_{\nu_n}[t^{D_n}]=\prod_{j=1}^{n-1}(1-\alpha_{n,j}+\alpha_{n,j}t).
\]

\begin{lemma}\label{lem:cgf}
Define the cumulant generating function of $D_n$ by
\[
K_n(t):=\log\E_{\nu_n}[e^{tD_n}] = \sum_{j=1}^{n-1}\log\bigl(1-\alpha_{n,j}+\alpha_{n,j}e^t\bigr).
\]
Then, for $|t|\le 1$, we have
%\begin{equation}\label{eq:K-third}
%|K_n^{(3)}(t)|\le Cn.
%\end{equation}
%Consequently,
\begin{align}
K_n(t)&=\bar d_n t+\frac12\sigma_n^2 t^2+O(n|t|^3),\label{eq:K-taylor}\\
{K_n'(t)}&{=\bar d_n+\sigma_n^2 t+O(n|t|^2),}\label{eq:Kprime-taylor}\\
{K_n''(t)}&{=\sigma_n^2+O(n|t|),}\label{eq:Ksecond-taylor}
\end{align}
where the implicit constant is uniform for all $|t|\le 1$.
\end{lemma}
\begin{proof}
For $\alpha\in[0,1]$, set $k_\alpha(t):=\log(1-\alpha+\alpha e^t)$. Note that $k_{\alpha}$ is smooth and the third derivative $k_\alpha^{(3)}(t)$ is uniformly bounded on the compact set $[0,1]\times[-1,1]$. Summing over $j=1,\dots,n-1$ and using Taylor's theorem with remainder, we get 
\[
K_n(t) = K_n'(0)t+\frac{1}{2}K_n''(0)t^2+ O(n|t|^3)\;. 
\]
Using Lemma~\ref{lem:des-moments}, we identify $\E_{\nu_n}[D_n]=K_n'(0)=\bar d_n$ and $\Var_{\nu_n}[D_n]=K_n''(0)=\sigma_n^2$. {The two derivative estimates follow from the same uniform bound on $K_n^{(3)}$ by Taylor's theorem applied to $K_n'$ and $K_n''$.} This completes the proof.
\end{proof}

Let $V_n=\vl(\pi)$ denote the number of valleys in a permutation $\pi\sim \nu_n$. We need the following easy estimates on the expectation and the variance of $V_n$, which can be obtained, for instance, from the bivariate generating function of $V_n$~\cite[Proposition 8, Equation (16)]{FS14Enum}.

\begin{lemma}
\label{lem:valley-moments}
Under the uniform measure $\nu_n$,
\[
\bar v_n=\E_{\nu_n}[V_n]=\frac{n-2}{3},
\qquad
\Var_{\nu_n}(V_n)=\frac{2n+2}{45}, \qquad \forall n\geq 4\;.
\]
%In particular, $\|V_n-\bar v_n\|_{L^2(\nu_n)}\asymp \sqrt n$.
\end{lemma}
%
%\begin{proof}
%Write
%\[
%V_n=\sum_{j=2}^{n-1}Y_j,
%\qquad
%Y_j:=\1_{\{\Pi_{n,j-1}>\Pi_{n,j}<\Pi_{n,j+1}\}},
%\]
%where $\Pi_n=(\Pi_{n,1},\dots,\Pi_{n,n})$ is uniform on $S_n$. For fixed $j$, among the six relative orders of $(\Pi_{n,j-1},\Pi_{n,j},\Pi_{n,j+1})$, exactly two have the middle entry smallest. Hence
%\[
%\E[Y_j]=\frac13,
%\qquad
%\Var(Y_j)=\frac13\left(1-\frac13\right)=\frac29.
%\]
%Therefore $\E[V_n]=(n-2)/3$.
%
%If $|i-j|\ge 3$, the indicators $Y_i$ and $Y_j$ depend on disjoint sets of coordinates and are independent. If $j=i+1$, the events $Y_i=1$ and $Y_{i+1}=1$ are incompatible, since the first requires $\Pi_{n,i}<\Pi_{n,i+1}$ while the second requires $\Pi_{n,i}>\Pi_{n,i+1}$. Thus
%\[
%\Cov(Y_i,Y_{i+1})=0-\frac19=-\frac19.
%\]
%If $j=i+2$, then $Y_i=Y_{i+2}=1$ is equivalent to the alternating pattern
%\[
%\Pi_{n,i-1}>\Pi_{n,i}<\Pi_{n,i+1}>\Pi_{n,i+2}<\Pi_{n,i+3}.
%\]
%Among the $5!=120$ relative orders of these five entries, exactly $16$ satisfy this pattern, so
%\[
%\Pbb(Y_i=Y_{i+2}=1)=\frac{16}{120}=\frac{2}{15},
%\]
%and consequently
%\[
%\Cov(Y_i,Y_{i+2})=\frac{2}{15}-\frac19=\frac1{45}.
%\]
%Summing all covariances gives
%\begin{align*}
%\Var(V_n)
%&=(n-2)\frac29+2(n-3)\left(-\frac19\right)+2(n-4)\left(\frac1{45}\right)\\
%&=\frac{2n+2}{45}.
%\qedhere
%\end{align*}
%\end{proof}

We note the following immediate corollary of Lemma~\ref{lem:valley-moments}. 
\begin{corollary}
\label{cor:exponentialMoment-valley}
Let $\eta_n$ be as in~\eqref{eqn:Etan}. Then, there exists a constant $C=C(c, p)>0$ such that \[|\eta_n(V_n(\pi)-\bar{v}_n)|\leq C,\]
for all $\pi\in S_n$. Furthermore,
\[
\norm{\exp(\eta_n(V_n-\bar{v}_n))-1}_{L^2(\nu_n)}=O_{c, p}(n^{-1/2}).
\]
\end{corollary}
\begin{proof}
Since $V_n(\pi)\leq n-1$ for all $\pi\in S_n$ and $\bar{v}_n=O(n)$, it follows that there is some constant $C$ such that $|V_n-\bar{v}_n|\leq Cn$ for all $\pi$. Since $\eta_n=O(n^{-1})$ from~\eqref{eqn:Etan}, it follows that $|\eta_n(V_n(\pi)-\bar{v}_n)|\leq C$ for some $C$.
%there exists a compact subset $K\subseteq \mathbb{R}$ such that $\eta_n(V_n(\pi)-\bar{v}_n)\in K$ for all $\pi\in S_n$. 
Since $e^{x}$ is Lipschitz on compact sets, there exists $L=L(C)>0$ such that $|e^{x}-1|\leq L|x|$ for all $|x|\leq C$. 
Therefore, 
\[
\norm{\exp(\eta_n(V_n-\bar{v}_n))-1}_{L^2(\nu_n)}\leq L{|\eta_n|}\norm{V_n-\bar{v}_n}_{L^2(\nu_n)}=O_{c, p}(n^{-1/2})\;.
\]
\end{proof}

\subsubsection{Some properties of \texorpdfstring{$J_{\pi}$}{Jpi}}
\begin{lemma}\label{lem:J-description}
There exists a constant $C=C(c, p)>0$ such that uniformly for all $\pi\in S_n$, we have
\begin{equation}
\label{eq:J-Uniform-bounds}
\Var_{\pi}(J_{\pi})\leq \E_\pi[J_\pi]\leq C\sqrt{n}\;.
\end{equation}
\end{lemma}
\begin{proof}
Using the definition of $F_{\pi}$ and $J_{\pi}$, we obtain that the generating function of $J_{\pi}$ is
\[
\frac{F_\pi(\tau_n z)}{F_\pi(\tau_n)}
=\left(\frac{1+p\tau_n z}{1+p\tau_n}\right)^{n-1-D_n(\pi)-V_n(\pi)}
\left(\frac{1+q\tau_n z}{1+q\tau_n}\right)^{D_n(\pi)-V_n(\pi)}
\left(\frac{1+\tau_n z}{1+\tau_n}\right)^{V_n(\pi)}.
\]
Each factor is the probability generating function of a Bernoulli random variable, which proves that $J_\pi$ is the sum of independent Bernoulli random variables consisting of:
\begin{itemize}
 \item $n-1-D_n(\pi)-V_n(\pi)$ Bernoullis with parameter $a_{p}$,
\item $D_n(\pi)-V_n(\pi)$ Bernoullis with parameter $a_q$,
 \item $V_n(\pi)$ Bernoulli with parameter $a_{1}$,
\end{itemize}
where \[a_r:=\frac{r\tau_n}{1+r\tau_n}, \qquad r\in [0, 1]\;.\] 

Since $\tau_n=O(n^{-1/2})$, Taylor expansion gives $a_r=O_{c, r}(n^{-1/2})$ where the implied constant only depends on $c, r$.  Since $J_{\pi}$ is a sum of at most $n$ Bernoullis with parameters $O(n^{-1/2})$, it follows that for every $\pi\in S_n$, we have
\[\Var{_{\pi}}(J_{\pi})\leq \mathbb{E}{_{\pi}}(J_{\pi})= {O_{c,p}(\sqrt n)},\]
where the implied constant can be chosen depending only on $c, p$. Since the variance of a Bernoulli sum is bounded by its mean, we conclude~\eqref{eq:J-Uniform-bounds}.
\end{proof}

Let $\bar d_n:=\mathbb{E}_{\nu_n}[\des(\pi)]$ and $\bar v_n = \mathbb{E}_{\nu_n}[\vl(\pi)]$ be the expected number of descents and the expected number of valleys in a uniform random permutation. 

\begin{lemma}\label{lem:J-centered-mean}
Let
\[
\bar m_n:=(n-1-\bar d_n-\bar v_n)a_p+(\bar d_n-\bar v_n)a_q+\bar v_n a_1.
\]
%Then, for every $\pi\in S_n$,
%\begin{equation}\label{eq:EJpi-center}
%\E_\pi[J_\pi]-\bar m_n
%=
%\bigl(D_n(\pi)-\bar d_n\bigr)(a_q-a_p)
%+
%\bigl(V_n(\pi)-\bar v_n\bigr)(a_1-a_p-a_q).
%\end{equation}
Then,
\begin{equation}\label{eq:EJpi-L2}
\bigl\|\E_\pi[J_\pi]-\bar m_n\bigr\|_{L^2(\nu_n)}=O_{c,p}(1).
\end{equation}
\end{lemma}
\begin{proof}
Using the Bernoulli decomposition of $J_{\pi}$ and the {definition of $\bar m_n$}, {we have}
\[
{\E_\pi[J_\pi]-\bar m_n=(D_n(\pi)-\bar d_n)(a_q-a_p)+(V_n(\pi)-\bar v_n)(a_1-a_p-a_q).}
\]
{Hence}
\[
{\bigl\|\E_\pi[J_\pi]-\bar m_n\bigr\|_{L^2(\nu_n)}\le |a_q-a_p|\norm{D_n-\bar d_n}_{L^2(\nu_n)}+ |a_1-a_p-a_q|\norm{V_n-\bar v_n}_{L^2(\nu_n)}\;.}
\]
Next, we observe that 
\begin{align*}
a_q-a_p &=(q-p)\tau_n+O(\tau_n^2)=O_{c,p}(n^{-1/2}),\\
a_1-a_p-a_q&=-2pq\tau_n^2+O(\tau_n^3)=O_{c,p}(n^{-1}).
\end{align*}
Using Lemmas~\ref{lem:des-moments} and{~}\ref{lem:valley-moments},
\begin{align*}
\bigl\|\E_\pi[J_\pi]-\bar m_n\bigr\|_{L^2(\nu_n)}
&\le |a_q-a_p|\,\|D_n-\bar d_n\|_{L^2(\nu_n)}
+|a_1-a_p-a_q|\,\|V_n-\bar v_n\|_{L^2(\nu_n)}\\
&=O(n^{-1/2})\cdot O(\sqrt n)+O(n^{-1})\cdot O(\sqrt n)
=O(1).
\qedhere
\end{align*}
\end{proof}

\begin{lemma}\label{lem:J-tail}
For every $B>0$ large enough, there exists $\gamma_B>0$ such that uniformly in $\pi\in S_n$,
\[
\Pbb_\pi(J_\pi>B\sqrt n)\le e^{-\gamma_B\sqrt n}
\]
for all large $n$.
\end{lemma}

\begin{proof}
Recall from Lemma~\ref{lem:J-description} that $J_\pi$ is a sum of independent Bernoulli variables. There exists $C=C(c, p)>0$ such that
\[
\E_\pi[e^{J_\pi}]\le \exp\!\bigl((e-1)\E_\pi[J_\pi]\bigr)\leq \exp(C\sqrt{n}),
\]
where the first inequality uses $1+x(e-1)\leq \exp((e-1)x)$ for $x\in[0, 1]$ and the second inequality uses {$\E_\pi[J_\pi]\le C\sqrt n$} from \eqref{eq:J-Uniform-bounds}. Fix $B>C$, Markov's inequality gives
\[
\Pbb_\pi(J_\pi>B\sqrt n)
\le e^{-B\sqrt n}\E_\pi[e^{J_\pi}]
\le e^{-(B-C)\sqrt n}.
\]
\end{proof}

\subsubsection{Some properties of \texorpdfstring{$R_{n,J_{\pi}}$}{RnJpi}}

\begin{lemma}\label{lem:R-discrete}
Fix $B>0$. There exist positive constants $c_B,C_B$ such that for all large $n$ and all integers $0\le j\le B\sqrt n$,
\begin{equation}\label{eq:R-bounds-1}
c_B\le R_{n,j}\le 1,
%\qquad
%|R_{n,j+1}-R_{n,j}|\le \frac{C_B}{\sqrt n},
%\qquad
%|R_{n,j+2}-2R_{n,j+1}+R_{n,j}|\le \frac{C_B}{n}.
\end{equation}
For all $0\le j,k\le B\sqrt n$, we have
\begin{equation}\label{eq:R-lipschitz}
|R_{n,j}-R_{n,k}|\le \frac{C_B}{\sqrt n}|j-k|,
\end{equation}
and
\begin{equation}\label{eq:R-taylor-discrete}
\bigl|R_{n,j}-R_{n,k}-(j-k)(R_{n,k+1}-R_{n,k})\bigr|
\le \frac{C_B}{n}(j-k)^2.
\end{equation}
\end{lemma}

\begin{proof}
The upper bound $R_{n,j}\le 1$ is immediate from the definition of $R_{n, j}$. For the lower bound, write
\[
-\log R_{n,j}
=\sum_{\ell=1}^{j-1}\bigl(-\log(1-\ell/n)+\log(1+\ell/M_n)\bigr).
\]
Since $\ell\leq j\le B\sqrt n$, we have $\ell/n\le B/\sqrt n\le 1/2$ for all large $n$. Therefore,
\[
-\log(1-\ell/n)+\log(1+\ell/M_n)
\le C\left(\frac{\ell}{n}+\frac{\ell}{M_n}\right)\leq C\frac{\ell}{n},
\]
with some absolute constant $C>0$. Summing over $\ell$ and using the fact that $j\leq C\sqrt{n}$, we obtain $-\log R_{n,j}\le C_B$. Exponentiating yields{~}\eqref{eq:R-bounds-1}.

Now set
\[
f_{n,j}:=\frac{R_{n,j+1}}{R_{n,j}}=\frac{1-j/n}{1+j/M_n}.
\]
Observe that for $j\le B\sqrt n$, we have
\begin{equation}
\label{eqn:FirstDifference}
    |f_{n,j}-1|
=\left|\frac{1-j/n}{1+j/M_n}-1\right|
\le \left(\frac{j}{n}+\frac{j}{M_n}\right)
\le \frac{C_B}{\sqrt n}.
\end{equation}

Therefore,  
\[|R_{n,j+1}-R_{n,j}|=|R_{n,j}||(f_{n,j}-1)|\leq \frac{C_B}{\sqrt{n}}\;.\]
Equation~\eqref{eq:R-lipschitz} follows by writing $|R_{n, j}-R_{n, k}|$ as a telescoping sum and applying the triangle inequality. 

Similarly,
\[
R_{n,j+2}-2R_{n,j+1}+R_{n,j}
=R_{n,j}\Bigl((f_{n,j}-1)(f_{n,j+1}-1)+(f_{n,j+1}-f_{n,j})\Bigr).
\]
The product term is $O_B(n^{-1})$ by the first difference estimate~\eqref{eqn:FirstDifference}. A direct calculation with the explicit formula for $f_{n,j}$ gives $|f_{n,j+1}-f_{n,j}|\le C_B/n$ for $j\le B\sqrt n$. This yields
\begin{equation}
    \label{eqn:SecondDifference}
    |R_{n,j+2}-2R_{n,j+1}+R_{n,j}|\le \frac{C_B}{n}\;.
\end{equation}
To see \eqref{eq:R-taylor-discrete} (assume for definiteness that $j>k$),{ write}
\[
R_{n,j}-R_{n,k}-(j-k)(R_{n,k+1}-R_{n,k})
=\sum_{u=k}^{j-1}\sum_{\ell=k}^{u-1}(R_{n,\ell+2}-2R_{n,\ell+1}+R_{n,\ell})\;.
\]
The bound~\eqref{eq:R-taylor-discrete} follows from the second difference estimate~\eqref{eqn:SecondDifference}. 

\end{proof}

\begin{lemma}\label{lem:R-factor-constant}
Let $\bar m_n$ be as in Lemma~\ref{lem:J-centered-mean}
\[
\bar m_n=(n-1-\bar d_n-\bar v_n)a_p+(\bar d_n-\bar v_n)a_q+\bar v_n a_1.
\]
Set $r_n:=R_{n,\lfloor \bar m_n\rfloor}$. Then,
\begin{equation}\label{eq:R-L2-final}
\left\|\frac{\E_\pi[R_{n,J_\pi}]}{r_n}-1\right\|_{L^2(\nu_n)}=O_{c,p}(n^{-1/2}).
\end{equation}
\end{lemma}

\begin{proof}
Throughout the proof, we assume that $n$ and $B>0$ are large so that Lemma~\ref{lem:J-tail} holds and, in addition, $\E_\pi[J_\pi]\le \tfrac12 B\sqrt n$ for every $\pi\in S_n$, which is possible by \eqref{eq:J-Uniform-bounds}. Set $j_\pi:=\lfloor \E_\pi[J_\pi]\rfloor$. Note that $0\le j_\pi\le B\sqrt n$.

We now show that $\E_{\pi}[R_{n, J_{\pi}}]=R_{n, j_{\pi}}+O_{c, p}({n^{-1/2}})$ for all $\pi\in S_n$. To this end, we first observe that for any $\pi\in S_n$, we have
\begin{align*}
    {\left|\E_{\pi}[R_{n, J_{\pi}}]-\E_{\pi}[R_{n, J_{\pi}}; J_{\pi}\leq B\sqrt{n}]\right|\leq \Pbb_{\pi}(J_{\pi}> B\sqrt{n})\;.}
\end{align*}
By Lemma~\ref{lem:J-tail}, the probability of the event $\{J_\pi>B\sqrt n\}$ is exponentially small and $R_{n, j}\leq 1$ for any $j$; {hence this error is $O(e^{-\gamma_{B}\sqrt n})$} for some constant $\gamma_{B}>0$. 

On the event $\{J_\pi\le B\sqrt n\}$, use \eqref{eq:R-taylor-discrete} with $k=j_\pi$ and $j=J_\pi$ to obtain
\[
R_{n,J_\pi}
=R_{n,j_\pi}+(J_\pi-j_\pi)(R_{n,j_\pi+1}-R_{n,j_\pi})+\varepsilon_\pi,\qquad |\varepsilon_\pi|\le \frac{C}{n}(J_\pi-j_\pi)^2.
\]
Taking $\E_\pi$ and using the fact $0\le \E_\pi[J_\pi]-j_\pi<1$, Lemma~\ref{lem:R-discrete}, and Equation~\eqref{eq:J-Uniform-bounds}, we get
\begin{align*}
\E_\pi[R_{n,J_\pi};J_\pi\le B\sqrt n]
%&=R_{n,j_\pi}+O\!\left(\frac{1}{\sqrt n}\right)+O\!\left(\frac{\Var_\pi(J_\pi)+1}{n}\right)\\
&=R_{n,j_\pi}+O_{c,p}(n^{-1/2})\;.
\end{align*}
% Hence
% \begin{equation}\label{eq:R-vs-jpi-clean}
% \E_\pi[R_{n,J_\pi}]=R_{n,j_\pi}+O_{c,p}(n^{-1/2})
% \end{equation}
% uniformly in $\pi$.

Finally, we compare $R_{n,j_\pi}$ with $r_n=R_{n,\lfloor \overline{m}_n\rfloor}$. To this end, we use \eqref{eq:R-lipschitz},
\[
|R_{n,j_\pi}-r_n|
\le \frac{C}{\sqrt n}\,|j_\pi-\lfloor \bar m_n\rfloor|
\le \frac{C}{\sqrt n}\Bigl(1+|\E_\pi[J_\pi]-\bar m_n|\Bigr).
\]
Taking $L^2(\nu_n)$ norms and using Lemma~\ref{lem:J-centered-mean} shows that
\[
\|R_{n,j_\pi}-r_n\|_{L^2(\nu_n)}=O_{c,p}(n^{-1/2}).
\]
Since $r_n$ is bounded away from $0$ and $\infty$ by Lemma~\ref{lem:R-discrete}, we conclude the proof.
\end{proof}

\subsubsection{Proof of Proposition~\ref{prop:descent_tilt_measure}}
\label{sec:FirstMainLemma}
In this section, we prove Proposition~\ref{prop:descent_tilt_measure}. The proof relies on the following two {lemmas}. 

\begin{lemma}\label{lem:Ltheta_Bound}
Let $L_n^{\theta_n}$ be as in Proposition~\ref{prop:descent_tilt_measure}. {Then there} exists a constant $C=C(c,p)$ such that for all large $n$ we have
\[
\E_{\nu_n}\bigl[(L_n^{{\theta_n}})^2\bigr]\le C\;.
\]
\end{lemma}
\begin{proof}
By definition,
\[
\E_{\nu_n}\bigl[(L_n^{{\theta_n}})^2\bigr]
=\frac{\E_{\nu_n}[e^{2\theta_n D_n}]}{\E_{\nu_n}[e^{\theta_n D_n}]^2}
=\exp\bigl\{K_n(2\theta_n)-2K_n(\theta_n)\bigr\}.
\]
Lemma~\ref{lem:cgf} yields $K_n(2\theta_n)-2K_n(\theta_n)
=\sigma_n^2\theta_n^2+O(n|\theta_n|^3)$.
Since $\sigma_n^2=O(n)$ and $\theta_n=O_{c, p}(n^{-1/2})$, {it follows} that $K_n(2\theta_n)-2K_n(\theta_n)=O_{c, p}(1)$.
\end{proof}

\begin{lemma}
\label{lem:HnConcentration}
Write 
\[
\mathcal H_n(\pi):=\exp\!\bigl(\eta_n(V_n(\pi)-\bar v_n)\bigr)\,\frac{\E_\pi[R_{n,J_\pi}]}{r_n}{.}\]
Then, 
\[\norm{\mathcal{H}_n(\pi)-1}_{L^2(\nu_n)}=O_{c, p}(n^{-1/2})\;.\]
\end{lemma}
\begin{proof}
Let $C>0$ be such that $\exp\!\bigl(\eta_n(V_n(\pi)-\bar v_n)\bigr)\leq C$ for all $\pi\in S_n$. Then,
\begin{align*}
\norm{\mathcal{H}_n(\pi)-1}_{L^2(\nu_n)} &\leq C\norm{\frac{\E_\pi[R_{n,J_\pi}]}{r_n}-1}_{L^2(\nu_n)}+  \norm{\exp\!\bigl(\eta_n(V_n(\pi)-\bar v_n)\bigr)-1}_{L^2(\nu_n)}\;.
\end{align*}
The proof is complete by observing that $\norm{\frac{\E_\pi[R_{n,J_\pi}]}{r_n}-1}_{L^2(\nu_n)}=O(n^{-1/2})$ by Lemma~\ref{lem:R-factor-constant}, and $\norm{\exp\!\bigl(\eta_n(V_n(\pi)-\bar v_n)\bigr)-1}_{L^2(\nu_n)}=O(n^{-1/2})$ by Corollary~\ref{cor:exponentialMoment-valley}.
\end{proof}

\begin{proof}[Proof of Proposition~\ref{prop:descent_tilt_measure}]
Write 
\[
\mathcal H_n(\pi):=\exp\!\bigl(\eta_n(V_n(\pi)-\bar v_n)\bigr)\,\frac{\E_\pi[R_{n,J_\pi}]}{r_n}, \quad \text{and}\quad \beta_n=\frac{1}{\E_{\mu_n^{\theta_n}}[\mathcal H_n]}.
\]
Then, $L_n(\pi)=\beta_n L_n^{{\theta_n}}(\pi)\,\mathcal H_n(\pi)$. Therefore, 
\begin{align*}
\TV\!\bigl(\nu^{(p)}_{n,m_n},\mu_n^{\theta_n}\bigr) &= \frac{1}{2}\norm{L_n-L_n^{\theta_n}}_{L^{1}(\nu_n)} \leq \frac{1}{2}\norm{L_n^{\theta_n}}_{L^2(\nu_n)} \norm{\beta_n\mathcal{H}_n-1}_{L^2(\nu_n)}\;.
\end{align*}
Note that $\norm{L_n^{\theta_n}}_{L^2(\nu_n)}$ {is bounded} by Lemma~\ref{lem:Ltheta_Bound}. On the other hand, we note that
\[\norm{\beta_n\mathcal{H}_n-1}_{L^2(\nu_n)} \leq |\beta_n|\norm{\mathcal{H}_n-1}_{L^2(\nu_n)}+ |\beta_n-1|\;.\]

The proof now follows by noting that 
\[
\norm{\mathcal{H}_n-1}_{L^2(\nu_n)} = O_{c, p}(n^{-1/2}), \quad \text{and}\quad \beta_n = 1+ O_{c, p}(n^{-1/2})\;.
\]
{The estimate for $\beta_n$ follows from $\beta_n^{-1}=\E_{\mu_n^{\theta_n}}[\mathcal H_n]$ and Cauchy--Schwarz, using Lemmas~\ref{lem:HnConcentration} and~\ref{lem:Ltheta_Bound}.}
\end{proof}

%%%%%%%%%%%%%%%%%%%%%%%%%%%%%%%%%%%%%%%%%%%%%%%%%%%%%%%%%%%%%%%%%%%%%%%%%%%%%%%%%%%%%%%%%%%%%%%%%%%%%%%%%%%%%%%%%%

\subsubsection{Proof of Proposition~\ref{prop:GaussianProfile}}
\label{sec:SecondMainLemma}
\begin{lemma}\label{lem:tilted-descents}
Under $\mu_n^{\theta_n}$, the random variable $D_n$ has the same law as the sum of independent Bernoulli random variables and 
\[
{\E_{\mu_n^{\theta_n}}(\des(\pi)) = \bar d_n+\sigma_n^2\theta_n+O_{c,p}(1), \qquad \Var_{\mu_n^{\theta_n}}(\des(\pi)) = \sigma_n^2+O_{c,p}(\sqrt n)\;.}
\]
\end{lemma}

\begin{proof}
We compute the moment generating function of $\des(\pi)$ under $\mu_{n}^{\theta_n}$. Note that 
\begin{align}
\label{eqn:MGF-under-tilt}
   \E_{\mu_n^{\theta_n}}[e^{t\des(\pi)}] &= \frac{\E_{\nu_n}[e^{(t+\theta_n)\des(\pi)}]}{\mathbb{E}_{\nu_n}[e^{\theta_n\des(\pi)}]}\;.
\end{align}
Recall that
\[
\E_{\nu_n}[t^{D_n}]=\prod_{j=1}^{n-1}(1-\alpha_{n,j}+\alpha_{n,j}t).
\]
Therefore, the moment generating function of $\des(\pi)$ under $\mu_n^{\theta_n}$ is
\[
{\prod_{j=1}^{n-1}\frac{1-\alpha_{n,j}+\alpha_{n,j}e^{\theta_n+t}}{1-\alpha_{n,j}+\alpha_{n,j}e^{\theta_n}}
=\prod_{j=1}^{n-1}\left(1-\widetilde\alpha_{n,j}+\widetilde\alpha_{n,j}e^t\right),}
\]
{where $\widetilde\alpha_{n,j}=\alpha_{n,j}e^{\theta_n}/(1-\alpha_{n,j}+\alpha_{n,j}e^{\theta_n})$. This is the moment generating function of the sum of independent Bernoulli variables.} Finally, from~\eqref{eqn:MGF-under-tilt}, we compute the cumulant generating function 
\[
\widetilde{K}_n(t):=\log \E_{\mu_n^{\theta_n}}[e^{t\des(\pi)}] = K_n(t+\theta_n)-K_n(\theta_n)\;.
\]
In particular, {using \eqref{eq:Kprime-taylor} and \eqref{eq:Ksecond-taylor}},
\begin{align*}
    \E_{\mu_n^{\theta_n}}[\des(\pi)] &=\widetilde{K}_n'(0)=K_n'(\theta_n) = {\bar{d}_n+ \sigma_n^2\theta_n+O_{c, p}(n|\theta_n|^2) = \bar d_n+\sigma_n^2\theta_n+O_{c,p}(1),} \\
    \Var_{\mu_n^{\theta_n}}[\des(\pi)] &= {\widetilde{K}_n''(0)} = K_n''(\theta_n) = {\sigma_n^2+O_{c, p}(n|\theta_n|) = \sigma_n^2+O_{c,p}(\sqrt n)}\;.
\end{align*}
\end{proof}

Since $\des(\pi)$ is a sum of independent Bernoulli random variables both under $\mu_n^{\theta_n}$ and $\nu_n$, we immediately get a Berry-Esseen bound for the Kolmogorov distance between $\des(\pi)$ (suitably centered and scaled) and the standard Gaussian. 

\begin{lemma}\label{lem:BE-descents}
There exists a constant $C=C(c,p)$ such that
\begin{equation}\label{eq:BE-uniform}
\sup_{x\in\R}\left|\Pbb_{\nu_n}\!\left(\frac{D_n-\bar d_n}{\sigma_n}\le x\right)-\PhiStd(x)\right|\le \frac{C}{\sqrt n},
\end{equation}
and
\begin{equation}\label{eq:BE-tilted}
\sup_{x\in\R}\left|\Pbb_{\mu_n^{\theta_n}}\!\left(\frac{D_n-K_n'(\theta_n)}{\sqrt{K_n''(\theta_n)}}\le x\right)-\PhiStd(x)\right|\le \frac{C}{\sqrt n}
\end{equation}
for all large $n$.
\end{lemma}

\begin{proof}
By Lemma~\ref{lem:tilted-descents}, both $D_n$ under $\nu_n$ and $D_n$ under $\mu_n^{\theta_n}$ are sums of independent Bernoulli variables. The classical Berry--Esseen theorem for sums of independent Bernoulli random variables, therefore, yields
\[
\sup_{x\in\R}\left|\Pbb\!\left(\frac{S-\E S}{\sqrt{\Var(S)}}\le x\right)-\PhiStd(x)\right|
\le C_{\mathrm{BE}}\,\frac{\sum_j \E|X_j-\E X_j|^3}{(\Var(S))^{3/2}}\leq \frac{C}{\sqrt{\Var(S)}},
\]
for some universal constant $C>0$.
%For a Bernoulli random variable $X$, one has $\E|X-\E X|^3\le \Var(X)$. Hence, the numerator is at most a constant multiple of the total variance, and the right-hand side is bounded by $C/\sqrt{\Var(S)}$. 
Under $\nu_n$, Lemma~\ref{lem:des-moments} gives $\Var(D_n)=\sigma_n^2\asymp n$. Under $\mu_n^{\theta_n}$, Lemma~\ref{lem:tilted-descents} gives $\Var(D_n)=K_n''(\theta_n)\asymp n$. This proves \eqref{eq:BE-uniform} and \eqref{eq:BE-tilted}.

\end{proof}

%%%%%%%%%%%%%%%%%%%%%%%%%%%%%%%%%%%%%%%%%%%%%%%%%%%%
%%%%%%%%%%%%%%%%%%%%%%%%%%%%%%%%%%%%%%%%%%%%%%%%%%%%
\begin{proof}[Proof of Proposition~\ref{prop:GaussianProfile}]
Set $\lambda:=-\frac{2p-1}{c\sqrt{12}}$ {and} let $\lambda_n:=\sigma_n\theta_n$. Using the definition of $\sigma_n$ and $\theta_n$, we see that
\begin{equation}\label{eq:lambda-rate-clean}
\lambda_n=\lambda+O_{c,p}(n^{-1/2}).
\end{equation}
Since the likelihood ratio 
\[
\frac{d\mu_n^{\theta_n}}{d\nu_n}(\pi)=e^{\theta_n D_n(\pi)-K_n(\theta_n)},
\]
is a monotone function of $D_n$, 
\begin{equation*}\label{eq:TV-threshold}
\TV\!\bigl(\mu_n^{\theta_n},\nu_n\bigr)
=
\left|\mu_n^{\theta_n}(A_n)-\nu_n(A_n)\right|,
\end{equation*}
where $A_n=\{\frac{d\mu_n^{\theta_n}}{d\nu_n}(\pi)\geq 1\}$. Set $t_n:=\frac{K_n(\theta_n)}{\theta_n}$ and note that
\[
A_n
=
\begin{cases}
\{D_n\ge \lceil t_n\rceil\}, & \theta_n>0,\\[0.4em]
\{D_n\le \lfloor t_n\rfloor\}, & \theta_n<0.
\end{cases}
\]
Using \eqref{eq:K-taylor}, we obtain
%\[
%K_n(\theta_n)=\bar d_n\theta_n+\frac12\sigma_n^2\theta_n^2+O(n|\theta_n|^3),
%\]
%and therefore
\begin{equation*}\label{eq:tn-location}
t_n=\bar d_n+\frac12\sigma_n^2\theta_n+O_{c,p}(1).
\end{equation*}
We will now compute $\nu_n(A_n)$ and $\mu_n^{\theta_n}(A_n)$. To this end, we first assume that $\theta_n>0$ and use Lemma~\ref{lem:BE-descents} and the fact that 
\[
\frac{t_n-\bar d_n}{\sigma_n}=\frac{\lambda_n}{2}+O_{c,p}(n^{-1/2})
\]
to conclude that 
\[
{\frac{t_n-K_n'(\theta_n)}{\sqrt{K_n''(\theta_n)}}=-\frac{\lambda_n}{2}+O_{c,p}(n^{-1/2}).}
\]
{Therefore,}
\begin{align*}
\nu_n(A_n) &= 1-\Phi\left(\frac{\lambda_n}{2}\right)+O_{c, p}(n^{-1/2}), \\
\mu_{n}^{\theta_n} (A_n)&= 1-\Phi\left(-\frac{\lambda_n}{2}\right)+O_{c, p}(n^{-1/2})\;.
\end{align*}
%Dividing by $\sigma_n$ gives
%\begin{equation}\label{eq:tn-uniform-standardized}
%\frac{t_n-\bar d_n}{\sigma_n}=\frac{\lambda_n}{2}+O_{c,p}(n^{-1/2}).
%\end{equation}
%Also, by \eqref{eq:Kprime-taylor} and \eqref{eq:Ksecond-taylor},
%\[
%K_n'(\theta_n)=\bar d_n+\sigma_n^2\theta_n+O_{c,p}(1),
%\qquad
%K_n''(\theta_n)=\sigma_n^2+O_{c,p}(\sqrt n).
%\]
%Therefore
%\begin{equation}\label{eq:tn-tilted-standardized}
%\frac{t_n-K_n'(\theta_n)}{\sqrt{K_n''(\theta_n)}}=-\frac{\lambda_n}{2}+O_{c,p}(n^{-1/2}).
%\end{equation}
Here, we used the fact that replacing $t_n$ by $\lceil t_n\rceil$ or $\lfloor t_n\rfloor$ changes these standardized quantities by at most $O(n^{-1/2})$.
% equations \eqref{eq:tn-uniform-standardized} and \eqref{eq:tn-tilted-standardized} remain valid for the actual threshold defining $A_n$.
%
Thus, we conclude that 
\[
\TV\!\bigl(\mu_n^{\theta_n},\nu_n\bigr)
=2\PhiStd\!\left(\frac{\lambda_n}{2}\right)-1+O_{c,p}(n^{-1/2})
=1-2\PhiStd\!\left(-\frac{\lambda_n}{2}\right)+O_{c,p}(n^{-1/2}).
\]
If $\theta_n<0$, then $A_n=\{D_n\le \lfloor t_n\rfloor\}$ and the same calculation gives
\[
\TV\!\bigl(\mu_n^{\theta_n},\nu_n\bigr)
=1-2\PhiStd\!\left(\frac{\lambda_n}{2}\right)+O_{c,p}(n^{-1/2})
=1-2\PhiStd\!\left(-\frac{|\lambda_n|}{2}\right)+O_{c,p}(n^{-1/2}).
\]
Combining the two cases,
\begin{equation}
\label{eq:TV-lambda-n}
\TV\!\bigl(\mu_n^{\theta_n},\nu_n\bigr)
=1-2\PhiStd\!\left(-\frac{|\lambda_n|}{2}\right)+O_{c,p}(n^{-1/2}).
\end{equation}
Finally, by \eqref{eq:lambda-rate-clean},
\[
\frac{|\lambda_n|}{2}=\frac{|2p-1|}{4\sqrt 3\,c}+O_{c,p}(n^{-1/2}),
\]
and since $\PhiStd$ is Lipschitz, replacing $|\lambda_n|/2$ by $|2p-1|/(4\sqrt 3\,c)$ changes the right-hand side of \eqref{eq:TV-lambda-n} by another $O_{c,p}(n^{-1/2})$. This completes the proof of Proposition~\ref{prop:GaussianProfile}.

\end{proof}

%\section{Cutoff in relative entropy, separation and \texorpdfstring{$\ell_{\infty}$}{linfinity}}

% The separation is closely related to the the $\ell_{\infty}$ distance 
% \[
% \norm{\mu-\nu_n}_{\infty} = \max_{\pi\in S_n}\left|1-\frac{\mu(\pi)}{\nu_n(\pi)}\right|\;.
% \]
%It can be easily checked that $\TV(\mu, \nu_n)\leq \sep(\mu)\leq \norm{\mu-\nu_n}_{\infty}$. 
%With slight modifications of the above proof, we can also obtain the cutoff profile for the asymmetric shelf-shuffle with respect to the relative entropy, separation distance, and the $\ell_{\infty}$ distance.

\subsection{Separation profile for asymmetric shelf shuffle}
% \begin{theorem}
% \label{thm:sep-cutoff}
% Fix $p\in(0,1)\setminus\{1/2\}$. Let $q=1-p$ and
% $r=\min\{p, q\}$. Fix $c>0$ and let $m_n=\lfloor c n^2\rfloor$. Then, as $n\to \infty$, we have
% \[
% \operatorname{sep}\!\left(\nu^{(p)}_{n,m_n},\nu_n\right)
% \longrightarrow
% 1-\exp\left\{-\frac{|2p-1|}{2c}\right\}.
% \]
% \end{theorem}

\begin{proof}[Proof of Theorem~\ref{thm:sep-cutoff}]
Recall that $L_{n,m}(\pi)=\frac{\nu^{(p)}_{n,m}(\pi)}{\nu_n(\pi)} = n!\nu^{(p)}_{n,m}(\pi)$ and note that
\[
\operatorname{sep}(\nu_{n, m}^{(p)})=1-\min_{\pi\in S_n}L_{n,m}(\pi)\;.
\]
It suffices to show that $\min_{\pi\in S_n}L_{n,m}(\pi)$ converges to $\exp\left\{-\frac{|2p-1|}{2c}\right\}$. To this end, we recall that for any $m\geq n$, we have
\[
L_{n,m}(\pi)
=
\frac{n!}{m^n}\binom{m}{n}F_\pi(\tau)\,\mathbb E_\pi[R_{n,J_\pi}]\;,
\]
where {$\tau = \frac{n}{m-n+1}$}.
% where
% \[
% A_{n,m}:=\frac{n!}{m^n}\binom{m}{n}
% =\prod_{\ell=0}^{n-1}\left(1-\frac{\ell}{m}\right),
% \]
% \(J_\pi\) has probability generating function
% \[
% \mathbb E_\pi[z^{J_\pi}]=\frac{F_\pi(\tau z)}{F_\pi(\tau)},
% \]
% and
% \[
% R_{n,j}:=\prod_{\ell=0}^{j-1}
% \frac{1-\ell/n}{1+\ell/M}.
% \]
We follow the same argument as in the proof of Proposition~\ref{prop:descent_tilt_measure}, but with $m=\lfloor c n^2\rfloor$ and therefore $\tau=\frac1{cn}+O_{c, p}(n^{-2})$. 

We first show that {$\mathbb E_{\pi}[R_{n, J_{\pi}}]=1+O_{c, p}(n^{-1})$} uniformly in $\pi$. The proof is similar to the one in the previous sections, {so} we skip the details and only sketch an outline. To this end, we use the fact that $m\asymp n^2$ to get  
\[
0\leq 1-R_{n, j}\leq C\min\{1, j^2/n\}, \qquad \text{for all } j\leq n-1
\]
for some universal constant $C>0$. Now recall from Lemma~\ref{lem:J-description} that $J_{\pi}$ is a sum of at most $(n-1)$ independent Bernoulli random variables {with parameters of the form $s\tau/(1+s\tau)$, where $s\in\{p,q,1\}$,} and $\tau =\frac{1}{cn}+O(n^{-2})$. Therefore,
\[
\sup_{\pi\in S_n}\mathbb E_\pi[J_\pi]=O(1),
\qquad
\sup_{\pi\in S_n}\mathbb E_\pi[J_\pi^2]=O(1).
\]
Arguing as in the proof of Lemma~\ref{lem:R-factor-constant}, with obvious modifications, one can show that
\[
\sup_{\pi\in S_n}
\left|\mathbb E_\pi[R_{n,J_\pi}]-1\right|
=O_{c, p}(n^{-1}).
\]
Therefore, asymptotically, it suffices to find the minimizer of $F_\pi(\tau)$. Suppose first that $p>q$. Let {$\rho_n=(n,n-1,\ldots,1)$} be the reverse permutation. Then $F_{\rho_n}(\tau)=(1+q\tau)^{n-1}${.}
Using the fact that $v\leq a$ for any permutation $\pi$, we conclude that {for} any $\pi\in S_n$ we have
\[
\frac{F_\pi(\tau)}{F_{\rho_n}(\tau)}
=\left(\frac{1+p\tau}{1+q\tau}\right)^{a-v}
\left(
\frac{1+\tau}{(1+q\tau)^2}
\right)^v >1\;,
\]
when $n$ is sufficiently large. 
{Therefore}, $\min_{\pi\in S_n}F_{\pi}(\tau)=F_{\rho_n}(\tau) = (1+q\tau)^{n-1}$ for $n$ sufficiently large. 
When $p<q$ and $n$ is sufficiently large, the same argument with the identity permutation $\iota_n=(1,2,\ldots,n)$ gives $\min_{\pi\in S_n} F_{\pi}(\tau) = F_{\iota_n}(\tau)=(1+p\tau)^{n-1}$.
We conclude that for $n$ sufficiently large, we have
\[
\min_{\pi\in S_n}F_\pi(\tau)=(1+r\tau)^{n-1},
\]
where $r=\min\{p, q\}$. Thus, 
\[
\min_{\pi\in S_n}L_{n,m}(\pi)
=
\frac{n!}{m^n}\binom{m}{n}(1+r\tau)^{n-1}\left(1+O(n^{-1})\right).
\]
Now set $A_{n,m}=\frac{n!}{m^n}\binom{m}{n}$ and observe that
\[
\log A_{n,m}
=
\sum_{\ell=0}^{n-1}\log\left(1-\frac{\ell}{m}\right)
=
-\frac{n(n-1)}{2m}+O_{c, p}\left(\frac{n^3}{m^2}\right)
=
-\frac1{2c}+O_{c, p}(n^{-1}),
\]
and
\[
(n-1)\log(1+r\tau)
=
\frac r c+O_{c, p}(n^{-1}).
\]
Therefore
\[
\min_{\pi\in S_n}L_{n,m}(\pi)
\longrightarrow
\exp\left\{-\frac1{2c}+\frac r c\right\}
=
\exp\left\{-\frac{1/2-r}{c}\right\}
=
\exp\left\{-\frac{|2p-1|}{2c}\right\}.
\]
This completes the proof.
\end{proof}

\subsection{Relative entropy profile for the asymmetric shelf shuffle}

\begin{proof}[Proof of Theorem~\ref{thm:entropy-cutoff}]
Recall that {$L_n(\pi)=\beta_nL_n^{\theta_n}(\pi)\mathcal H_n(\pi)$}. Write {$\widetilde{H}_n = \beta_n \mathcal H_n$} for notational convenience. Therefore, we can write 
\[
\Ent(\nu_{n, m}^{(p)}) = \Ent(\mu_{n}^{\theta_n})+ {\mathbb E_{\nu_n}}[(\widetilde{H}_n-1)L_n^{\theta_n}(\pi)\log L_n^{\theta_n}(\pi)]+ \mathbb{E}_{\nu_n}(L_n^{\theta_n}(\pi)\widetilde{H}_n(\pi)\log\widetilde{H}_n(\pi)).
\]
In Lemma~\ref{lem:Ltheta_Bound}, we show that $\mathbb{E}_{\nu_n}[(L_n^{\theta_n})^2]\leq C$ for some constant $C$. However, the proof is readily {adapted} to show that $\mathbb{E}_{\nu_n}[{(L_n^{\theta_n})^2(\log L_n^{\theta_n})^2}]\leq C$. Indeed, for $x\geq 1$, define 
\[
G_n(x) := \mathbb{E}_{\nu_n}[(L_n^{\theta_n})^{x}] = \exp\left(K_n(x\theta_n)-xK_n(\theta_n)\right)\;.
\]
Observe that $G_n''(2) = \mathbb{E}_{\nu_n}[{(L_n^{\theta_n})^2(\log L_n^{\theta_n})^2}]$. On the other hand, we note that 
\[
G_n''(2) = G_n(2)[\theta_nK_n'(2\theta_n)-K_n(\theta_n)]^2+ G_n(2)[\theta_n^2K_n''(2\theta_n)]\;.
\]
Using the Taylor expansion of $K_n$, it is easy to verify that 
\[{[\theta_nK_n'(2\theta_n)-K_n(\theta_n)]^2=O_{c, p}(1),\qquad \theta_n^2K_n''(2\theta_n)=O_{c,p}(1).}\] 
The desired claim therefore follows by noticing that $G_n(2)=\mathbb{E}_{\nu_n}[(L_n^{\theta_n})^2]\leq C$. On the other hand, the proof of Lemma~\ref{lem:HnConcentration} shows that $\norm{\widetilde{H}_n-1}_{L^2(\nu_n)}=\norm{{\beta_n\mathcal H_n}-1}_{L^2(\nu_n)}=O(n^{-1/2})$. Therefore, the Cauchy-Schwarz inequality gives
\[
|{\mathbb E_{\nu_n}}[(\widetilde{H}_n-1)L_n^{\theta_n}(\pi)\log L_n^{\theta_n}(\pi)]|= O_{c, p}(n^{-1/2})\;.
\]

We next show that $\mathbb{E}_{\nu_n}(L_n^{\theta_n}(\pi)\widetilde{H}_n(\pi)\log\widetilde{H}_n(\pi))\to 0$ as $n\to \infty$. To this end, we first note that 
\[
\mathbb{E}_{\mu_n^{\theta_n}}[|\widetilde{H}_n-1|] =\mathbb{E}_{\nu_n}[L_n^{\theta_n}|\widetilde{H}_n-1|] \leq \norm{L_n^{\theta_n}}_{L^2(\nu_n)}\norm{\widetilde{H}_n-1}_{L^2(\nu_n)}\to 0,
\]
as $n\to \infty$ thanks to Lemma~\ref{lem:HnConcentration} and Lemma~\ref{lem:Ltheta_Bound}. This proves that $\widetilde{H}_n$ converges to $1$ in probability with respect to $\mu_n^{\theta_n}$. Since $x\mapsto x\log x$ is continuous (we use the convention that $0\log 0=0$) and $\widetilde{H}_n$ is uniformly bounded, we conclude that 
\[
\mathbb{E}_{\nu_n}(L_n^{\theta_n}(\pi)\widetilde{H}_n(\pi)\log\widetilde{H}_n(\pi))=\mathbb{E}_{\mu_n^{\theta_n}}(\widetilde{H}_n(\pi)\log\widetilde{H}_n(\pi))\to 0,
\]
as $n\to \infty$. 

Therefore, it suffices to show that $\Ent(\mu_n^{\theta_n})\to \frac{(2p-1)^2}{24c^2}$ as $n\to \infty$. To this end, we observe that 
\begin{align*}
    \Ent(\mu_n^{\theta_n}) &=\E_{\nu_n}[L_n^{\theta_n}(\pi)\log L_n^{\theta_n}(\pi)]\\
    &=\E_{\nu_n}\left[ \frac{e^{\theta_n\des(\pi)}}{e^{K_n(\theta_n)}}\left(\theta_n\des(\pi)-K_n(\theta_n)\right)\right]\\
    &= \theta_nK_n'(\theta_n)-K_n(\theta_n)\;.
\end{align*}
Once again using the Taylor expansion of $K_n$ and the fact that $\theta_n=-\frac{2p-1}{c\sqrt{n}}+O(n^{-1})$ and $\sigma_n^2=(n+1)/12$, we conclude that 
\[
\theta_nK_n'(\theta_n)-K_n(\theta_n) = \frac{1}{2}\sigma_n^2\theta_n^2 = \frac{(2p-1)^2}{24c^2}+{O_{c,p}(n^{-1/2})}\;,
\]
which completes the proof.

\end{proof}

%%%%%%%%%%%%%%%%%%%%%%%%%%%%%%%

\bibliographystyle{alpha} % can use amsalpha but repeated authors are dashed
\bibliography{references}

@article {chen2025cutoff,
    AUTHOR = {Chen, Ray and Ottolini, Andrea},
     TITLE = {Cutoff in total variation for the shelf shuffle},
   JOURNAL = {Electron. Commun. Probab.},
  FJOURNAL = {Electronic Communications in Probability},
    VOLUME = {30},
      YEAR = {2025},
     PAGES = {Paper No. 44, 10},
      ISSN = {1083-589X},
   MRCLASS = {60J10},
  MRNUMBER = {4908782},
       DOI = {10.1214/25-ecp691},
       URL = {https://doi.org/10.1214/25-ecp691},
}

@unpublished{kuba,
  title={On Card guessing after a single shelf shuffle},
  author={Clay, Alexander  and Kuba, Markus and Tripathi, Raghavendra},
  year={2026},
  note={arXiv:2602.12928},
  publisher={arXiv}
}

@unpublished{clay2025guessing,
  title={Guessing Strategies for Shuffling Machines},
  author={Clay, Alexander},
  journal={arXiv preprint},
  note={arXiv:2507.10294},
  year={2025}
}

@unpublished{clay2025limit,
  title={Limit Theorems for Descents and Inversions of Shelf-Shuffles},
  author={Clay, Alexander},
  journal={arXiv preprint},
  note={arXiv:2510.00343},
  year={2025}
}

@unpublished{tripathi2026position,
  title={On the position matrix of single-shelf shuffle and card guessing},
  author={Tripathi, Raghavendra},
  journal={arXiv preprint},
  note={arXiv:2602.07920},
  year={2026}
}

@article {Pitman97,
    AUTHOR = {Pitman, Jim},
     TITLE = {Probabilistic bounds on the coefficients of polynomials with
              only real zeros},
   JOURNAL = {J. Combin. Theory Ser. A},
  FJOURNAL = {Journal of Combinatorial Theory. Series A},
    VOLUME = {77},
      YEAR = {1997},
    NUMBER = {2},
     PAGES = {279--303},
      ISSN = {0097-3165,1096-0899},
   MRCLASS = {05A16 (05A20 60C05 60G99)},
  MRNUMBER = {1429082},
MRREVIEWER = {Edward\ A.\ Bender},
       DOI = {10.1006/jcta.1997.2747},
       URL = {https://doi.org/10.1006/jcta.1997.2747},
}

@article {DiaconisFulmanHolmes2013,
    AUTHOR = {Diaconis, Persi and Fulman, Jason and Holmes, Susan},
     TITLE = {Analysis of casino shelf shuffling machines},
   JOURNAL = {Ann. Appl. Probab.},
  FJOURNAL = {The Annals of Applied Probability},
    VOLUME = {23},
      YEAR = {2013},
    NUMBER = {4},
     PAGES = {1692--1720},
      ISSN = {1050-5164,2168-8737},
   MRCLASS = {60C05 (05A15)},
  MRNUMBER = {3098446},
MRREVIEWER = {Kent\ E.\ Morrison},
       DOI = {10.1214/12-aap884},
       URL = {https://doi.org/10.1214/12-aap884},
}

@book {DiaconisFulman2023Shuffling,
    AUTHOR = {Diaconis, Persi and Fulman, Jason},
     TITLE = {The mathematics of shuffling cards},
 PUBLISHER = {American Mathematical Society, Providence, RI},
      YEAR = {[2023] \copyright 2023},
     PAGES = {xii+346},
      ISBN = {[9781470463038]},
   MRCLASS = {60-02 (05Axx 60B15 60J10 91A60)},
  MRNUMBER = {4565368},
MRREVIEWER = {Martin\ V.\ Hildebrand},
}

@article {bayer1992trailing,
    AUTHOR = {Bayer, Dave and Diaconis, Persi},
     TITLE = {Trailing the dovetail shuffle to its lair},
   JOURNAL = {Ann. Appl. Probab.},
  FJOURNAL = {The Annals of Applied Probability},
    VOLUME = {2},
      YEAR = {1992},
    NUMBER = {2},
     PAGES = {294--313},
      ISSN = {1050-5164,2168-8737},
   MRCLASS = {60C05 (20B30 60B15)},
  MRNUMBER = {1161056},
MRREVIEWER = {David\ J.\ Aldous},
       URL =
              {http://links.jstor.org/sici?sici=1050-5164(199205)2:2<294:TTDSTI>2.0.CO;2-F&origin=MSN},
}

@article {aldous1986shuffling,
    AUTHOR = {Aldous, David and Diaconis, Persi},
     TITLE = {Shuffling cards and stopping times},
   JOURNAL = {Amer. Math. Monthly},
  FJOURNAL = {American Mathematical Monthly},
    VOLUME = {93},
      YEAR = {1986},
    NUMBER = {5},
     PAGES = {333--348},
      ISSN = {0002-9890,1930-0972},
   MRCLASS = {60C05 (60B15 60G40 60J15)},
  MRNUMBER = {841111},
MRREVIEWER = {Endre\ Cs\'aki},
       DOI = {10.2307/2323590},
       URL = {https://doi.org/10.2307/2323590},
}

@unpublished{salez2025modern,
  title={Modern aspects of Markov chains: entropy, curvature and the cutoff phenomenon},
  author={Salez, Justin},
  journal={arXiv preprint},
  note={arXiv:2508.21055},
  year={2025}
}

@article {sellke2022cutoff,
    AUTHOR = {Sellke, Mark},
     TITLE = {Cutoff for the asymmetric riffle shuffle},
   JOURNAL = {Ann. Probab.},
  FJOURNAL = {The Annals of Probability},
    VOLUME = {50},
      YEAR = {2022},
    NUMBER = {6},
     PAGES = {2244--2287},
      ISSN = {0091-1798,2168-894X},
   MRCLASS = {60C05 (20P05 60B15 60J10)},
  MRNUMBER = {4499839},
MRREVIEWER = {Wojciech\ Bartoszek},
       DOI = {10.1214/22-aop1582},
       URL = {https://doi.org/10.1214/22-aop1582},
}

@article {FKLP,
    AUTHOR = {Fulman, Jason and Kim, Gene B. and Lee, Sangchul and Petersen,
              T. Kyle},
     TITLE = {On the joint distribution of descents and signs of
              permutations},
   JOURNAL = {Electron. J. Combin.},
  FJOURNAL = {Electronic Journal of Combinatorics},
    VOLUME = {28},
      YEAR = {2021},
    NUMBER = {3},
     PAGES = {Paper No. 3.37, 30},
      ISSN = {1077-8926},
   MRCLASS = {05A15 (60F05)},
  MRNUMBER = {4301305},
MRREVIEWER = {David\ Bevan},
       DOI = {10.37236/10222},
       URL = {https://doi.org/10.37236/10222},
}

@article {Alperen,
    AUTHOR = {\"Ozdemir, Alperen},
     TITLE = {Martingales and descent statistics},
   JOURNAL = {Adv. in Appl. Math.},
  FJOURNAL = {Advances in Applied Mathematics},
    VOLUME = {140},
      YEAR = {2022},
     PAGES = {Paper No. 102395, 34},
      ISSN = {0196-8858,1090-2074},
   MRCLASS = {60C05 (05A15 60G42)},
  MRNUMBER = {4443762},
MRREVIEWER = {David\ Bevan},
       DOI = {10.1016/j.aam.2022.102395},
       URL = {https://doi.org/10.1016/j.aam.2022.102395},
}

@article {FS14Enum,
    AUTHOR = {Fewster, Christopher J. and Siemssen, Daniel},
     TITLE = {Enumerating permutations by their run structure},
   JOURNAL = {Electron. J. Combin.},
  FJOURNAL = {Electronic Journal of Combinatorics},
    VOLUME = {21},
      YEAR = {2014},
    NUMBER = {4},
     PAGES = {Paper 4.18, 19},
      ISSN = {1077-8926},
   MRCLASS = {05A05 (05A15)},
  MRNUMBER = {3284067},
MRREVIEWER = {Darko\ Veljan},
       DOI = {10.37236/4235},
       URL = {https://doi.org/10.37236/4235},
}

\end{document}